\documentclass[12pt]{amsart}

\usepackage{amsmath,amsthm,amssymb}
\usepackage{verbatim}
\usepackage{tikz}
\usepackage{float}
\usepackage{mathtools}
\usepackage{subcaption}
\usepackage{algorithm}
\usepackage{algorithmic}
\usepackage{MnSymbol}
\usepackage{graphicx}
\usepackage[left=2.5cm,top=2.5cm,right=2.5cm,bottom=2.5cm]{geometry}

\newtheorem{theorem}{Theorem}
\newtheorem{lemma}[theorem]{Lemma}
\newtheorem{corollary}[theorem]{Corollary}

\theoremstyle{definition}

\newcommand{\qedclaim}{\hfill $\diamond$ \medskip}

\newcommand{\loc}[1]{\zeta(#1)}

\begin{document}

\title{The localization number and metric dimension of graphs of diameter 2}\thanks{The first and second authors were supported by NSERC}

\author[A.\ Bonato]{Anthony Bonato}
\author[M. A.\ Huggan]{Melissa A. Huggan}
\author[T.\ Marbach]{Trent Marbach}
\address[A1, A2, A3]{Ryerson University, Toronto, Canada}
\email[A1]{(A1) abonato@ryerson.ca}
\email[A2]{(A2) melissa.huggan@ryerson.ca}
\email[A3]{(A3) trent.marbach@ryerson.ca}

\begin{abstract}
We consider the localization number and metric dimension of certain graphs of diameter $2$, focusing on families of Kneser graphs and graphs without 4-cycles. For the Kneser graphs with diameter $2$, we find upper and lower bounds for the localization number and metric dimension, and in many cases these parameters differ only by an additive constant. Our results on the metric dimension
of Kneser graphs improve on earlier ones, yielding exact values in infinitely many cases. We determine bounds on the localization number and metric dimension of Moore graphs of diameter $2$ and polarity graphs.
\end{abstract}


\maketitle

\section{Introduction}

\emph{Graph searching} considers combinatorial models for the detection or neutralization of an adversary's activity on a graph. Such models often focus on vertex-pursuit games, where agents or cops are attempting to capture an adversary or robber loose on the vertices of a graph. The players move at alternating ticks of the clock, and have restrictions on their movements or relative speed depending on the game played. The most studied such game is Cops and Robbers, where the cops and robber can only move to vertices with which they share an edge. The cop number is the minimum number of cops needed to guarantee the robber's capture. How the players move and the rules of capture depend on which variant is studied. These variants are motivated by problems in practice or inspired by foundational issues in computer science, discrete mathematics, and artificial intelligence, such as robotics and network security. For a survey of graph searching, see~\cite{bp,by,fomin}, and see \cite{BN} for more background on Cops and Robbers.

We investigate the localization game and metric dimension in the present work. In the localization game, two players operate on a connected graph, with one player controlling a set of $k$ \emph{cops}, where $k$ is a positive integer, and the second controlling a single \emph{robber}. Unlike in Cops and Robbers, the cops play with imperfect information: the robber is invisible to the cops during gameplay. The game is played over a sequence of discrete time-steps; a \emph{round} is a cop move and a subsequent robber move. The robber occupies a vertex of the graph, and when the robber is ready to move during a round, he may move to a neighboring vertex or remain on his current vertex. A move for the cops is a placement of cops on a set of vertices. Note that the cops are not limited to moving to neighboring vertices. At the beginning of the game, the robber chooses his starting vertex. After this, the cops move first, followed by the robber; thereafter, they move on alternate time-steps. Observe that any subset of cops may move in a given round. In each round, the cops occupy a set of vertices $u_1, u_2, \ldots , u_k$ and each cop sends out a \emph{cop probe}, which gives their distance $d_i$, from $u_i$ to the robber, where $1\le i \le k$. Hence, in each round, the cops determine a \emph{distance vector} $(d_1, d_2, \ldots , d_k)$ of cop probes, which is unique up to the ordering of the cops. Note that relative to the placement of the cops, there may be more than one vertex with the same distance vector.
For example, in an $n$-vertex clique with a single cop, so long as the cop is not on the robber's vertex, there are $n-1$ such vertices.
The cops win if they have a strategy to determine, after finitely many rounds, the vertex the robber occupies, at which time we say that the cops {\em capture} the robber.
If the robber is not located, then the robber may move in the next round, and the cops may move to other vertices resulting in an updated distance vector. The robber wins if he is never captured.

For a connected graph $G$, define the \emph{localization number} of $G$, written $\loc{G}$, to be the least integer $k$ for which $k$ cops have a winning strategy over any possible strategy of the robber (that is, we consider the worst case for the cops in that the robber a priori knows the entire strategy of the cops).
As placing a cop on each vertex gives a distance vector containing a $0$, which corresponds to the location of the robber, $\loc{G}$ is at most $n$ and so is well-defined. The localization game was first introduced for one cop by Seager~\cite{seager1,seager2} and was further studied in, for example, \cite{BHM,BK,nisse1,BDELM,car,DEFMP,DFP,has}.

In \cite{nisse1}, Bosek et al.\ showed that $\loc{G}$ is bounded above by the pathwidth of $G$ and that the localization number is unbounded even on graphs obtained by adding a universal vertex to a tree. They also proved that computing $\loc{G}$ is \textbf{NP}-hard for graphs with diameter $2$, and they studied the localization game for geometric graphs. In \cite{DFP}, the localization number was studied for binomial random graphs with diameter $2$, with further work on the localization number of random graphs done in \cite{DEFMP}. Bonato and Kinnersley~\cite{BK} studied the localization number for graphs based on their degeneracy. In ~\cite{BK}, they resolved a conjecture of Bosek et al.~\cite{nisse1} relating $\loc{G}$ and the chromatic number; further, they proved that the localization number of outerplanar graphs is at most 2, and they proved an asymptotically tight upper bound on the localization number of the hypercube.
The localization number of the incidence graphs of designs was studied in \cite{BHM}. In particular, they gave exact values for the localization number of the incidence graphs of projective and affine planes, and bounds for the incidence graphs of Steiner systems and transversal designs.

The \emph{metric dimension} of a graph $G$, written $\beta(G)$ (also referred to as $\mu(G)$ and $\mathrm{dim}(G)$ in the literature), is the minimum number of cops needed in the localization game so that the cops can win in one round. Hence, $\loc{G} \le \beta(G)$, but in many cases this inequality is far from tight. Metric dimension was introduced in the 1970s by Slater \cite{slater} and, independently, by Harary and Melter \cite{hm}.
A \emph{resolving set} is a set of $\beta(G)$ vertices that the cops can play on to win the localization game in one round.
A survey on metric dimension and related concepts may be found in \cite{bc}. Graphs of diameter $2$ with metric dimension $2$ were characterized in \cite{kth}.

For a complete graph of order $n$ (that is, a graph of diameter 1), the metric dimension and localization numbers are equal to $n-1.$ For graphs of diameter $2$, where distance probes return either 0, 1, or 2, the determination of these parameters is a much more elusive problem. In this paper, we focus on the localization number of certain graphs of diameter $2$; in particular, the Kneser graphs of diameter $2$ and diameter 2 graphs without 4-cycles.

The first family we consider are Kneser graphs, which are a well-known family of non-intersection graphs. For integers $k,n \ge 1$ with $n > k,$ the \emph{Kneser graph} $K(k,n)$ has vertices labeled by the $k$-tuples on $[n] = \{ 1,2, \ldots, n \},$ with two vertices adjacent if and only if their vertex labels are disjoint. Kneser graphs were introduced by Lov\'asz \cite{lovasz} to resolve Kneser's conjecture on their chromatic number. In Section~\ref{sec:MetricDim}, we study the Kneser graphs that have diameter $2$ and find upper and lower bounds for the localization number and metric dimension of these graphs, which in many cases differ only by an additive constant; see Theorem~\ref{finall}. While the work on the localization number of these graphs is new, the metric dimension for Kneser graphs was previously studied in \cite{bcggmmp,bc}. Before this work, no asymptotically tight results were known for Kneser graphs for infinite families when $k$ is a fixed constant. For each fixed even $k\geq 4$, the results of this paper give the exact value of the metric dimension and localization number up to an additive constant of an infinite subclass of Kneser graphs. In particular, for a fixed even $k\geq 6$, Corollary~\ref{cordimd} and Lemma~\ref{lemdimd} provide that $\beta(K(k,n)) = n/2+n/k$ for an infinite number of values of $n$.

In Section~\ref{secc4}, we consider graphs of diameter $2$ with no $4$-cycles as subgraphs. As proven in \cite{bef}, there are three subclasses of graphs that have diameter $2$ and contain no $4$-cycle: graphs with maximum degree $n-1$, the Moore graphs, and the polarity graphs. We define the latter two graph families in Section~\ref{secc4}. The family of graphs with maximum degree $n-1$ and no $4$-cycles are the graph with a universal vertex $u$, which when removed leaves isolated vertices (that is, vertices of degree 0) or paths of length two.
The localization number of this graph will be $1$ if removing the universal vertex leaves $n-1$ isolated vertices, and $2$ otherwise. If $k$ is the number of connected components remaining when the universal vertex is deleted, then the metric dimension of this graph will be $k$, except in the case of $3$-vertex clique where it is $2.$ We consider the remaining two graph families. In Section~\ref{secc4}, we bound the metric dimension and  localization number of the Moore graphs of diameter $2$, including the Hoffman-Singleton graph. In particular, we show in Theorem~\ref{thm:MooreMD} that a $k$-regular Moore graph $G$ of diameter 2 has metric dimension $k \leq \beta(G) \leq 2k-3$ when $k\ge 3,$ and in Theorem~\ref{one} that $G$'s localization number is either $k-1$ or $k$ when $k\ge 5$.
We finish with Theorems~\ref{thm:dim_polarity} and \ref{finalt}, which together provide that if $G$ is a polarity graph of order $q^2+q+1$, where $q$ is a prime power, then $2q-5 \leq \beta(G) \leq 2q-1$ and $(2q-5)/3 \leq \zeta(G) \leq 2q-1$.

Throughout, all graphs considered are simple, undirected, connected, and finite. For a general reference for graph theory, see~\cite{West}. The \emph{closed neighborhood} of $u$, written $N[u],$ consists of a vertex $u$ along with neighbors of $u$. The \emph{second neighborhood} of $u$, written $N_2(u)$, are the vertices of distance 2 to $u$. We refer to vertices in $N_2(u)$ as \emph{second neighbors}. The \emph{distance} between vertices $u$ and $v$ is denoted by $d(u,v).$

\section{Kneser graphs}\label{sec:MetricDim}
A Kneser graph has diameter $2$ if and only if $n \geq 3k$, and we focus on this case. The Kneser graph $K(2,6),$ which is of diameter $2$, is depicted Figure~\ref{ex:Kne}.

\begin{figure}[hb]
\begin{center}
\begin{tikzpicture}[scale=0.7]
\path (0:5) node[circle, draw=black, fill=gray](x12) {12}
(24:5) node[circle, draw=black](x13) {13}
(2*24:5) node(x14)[circle, draw=black] {14}
(3*24:5) node(x15)[circle, draw=black] {15}
(4*24:5) node[circle, draw=black, fill=gray](x16) {16}
(5*24:5) node[circle, draw=black, fill=gray](x23) {23}
(6*24:5) node[circle, draw=black](x24) {24}
(7*24:5) node[circle, draw=black](x25) {25}
(8*24:5) node[circle, draw=black](x26) {26}
(9*24:5) node[circle, draw=black, fill=gray](x34) {34}
(10*24:5) node[circle, draw=black](x35) {35}
(11*24:5) node[circle, draw=black](x36) {36}
(12*24:5) node[circle, draw=black, fill=gray](x45) {45}
(13*24:5) node[circle, draw=black](x46) {46}
(14*24:5) node[circle, draw=black, fill=gray](x56) {56};
%distance - radius
%1 - 3.0
%2 - 2
%3 - 1.5
\draw (x12) .. controls (12*24:0.5) .. (x34);
\draw (x12) .. controls (12.5*24:0.5) .. (x35);
\draw (x12) .. controls (13*24:1) .. (x36);
\draw (x12) .. controls (13.5*24:1.5) .. (x45);
\draw (x12) .. controls (14*24:2) .. (x46);
\draw (x12) .. controls (14.5*24:3.0) .. (x56);
\draw (x13) .. controls (3.5*24:1) .. (x24);
\draw (x13) .. controls (4*24:0.5) .. (x25);
\draw (x13) .. controls (4.5*24:0) .. (x26);
\draw (x13) .. controls (14*24:1) .. (x45);
\draw (x13) .. controls (14.5*24:1.5) .. (x46);
\draw (x13) .. controls (0*24:2) .. (x56);
\draw (x14) .. controls (3.5*24:1.5) .. (x23);
\draw (x14) .. controls (4.5*24:1) .. (x25);
\draw (x14) .. controls (5*24:0.5) .. (x26);
\draw (x14) .. controls (6*24:0) .. (x35);
\draw (x14) .. controls (14*24:0.5) .. (x36);
\draw (x14) .. controls (0.5*24:1.5) .. (x56);
\draw (x15) .. controls (4*24:2) .. (x23);
\draw (x15) .. controls (4.5*24:1.5) .. (x24);
\draw (x15) .. controls (5.5*24:1) .. (x26);
\draw (x15) .. controls (6*24:0.5) .. (x34);
\draw (x15) .. controls (14.5*24:0) .. (x36);
\draw (x15) .. controls (0.5*24:1) .. (x46);
\draw (x16) .. controls (4.5*24:3) .. (x23);
\draw (x16) .. controls (5*24:2) .. (x24);
\draw (x16) .. controls (5.5*24:1.5) .. (x25);
\draw (x16) .. controls (6.5*24:1) .. (x34);
\draw (x16) .. controls (7*24:0.5) .. (x35);
\draw (x16) .. controls (0.5*24:0) .. (x45);
%\draw (x23) -- (x14);
%\draw (x23) -- (x15);
%\draw (x23) -- (x16);
\draw (x23) .. controls (8.5*24:0) .. (x45);
\draw (x23) .. controls (1.5*24:0) .. (x46);
\draw (x23) .. controls (2*24:0.5) .. (x56);
%\draw (x24) -- (x13);
%\draw (x24) -- (x15);
%\draw (x24) -- (x16);
\draw (x24) .. controls (8*24:1) .. (x35);
\draw (x24) .. controls (8.5*24:0.5) .. (x36);
\draw (x24) .. controls (2.5*24:0) .. (x56);
\draw (x25) .. controls (8*24:2) .. (x34);
\draw (x25) .. controls (9*24:1) .. (x36);
\draw (x25) .. controls (10*24:0.5) .. (x46);
\draw (x26) .. controls (8.5*24:3) .. (x34);
\draw (x26) .. controls (9*24:2) .. (x35);
\draw (x26) .. controls (10*24:1) .. (x45);
\draw (x34) .. controls (11*24:1) .. (x56);
\draw (x35) .. controls (11.5*24:1.5) .. (x46);
\draw (x36) .. controls (11.5*24:3) .. (x45);

\end{tikzpicture}
\caption{The Kneser graph $K(2,6)$\label{ex:Kne} with a resolving set as shaded circles.}
\end{center}
\end{figure}
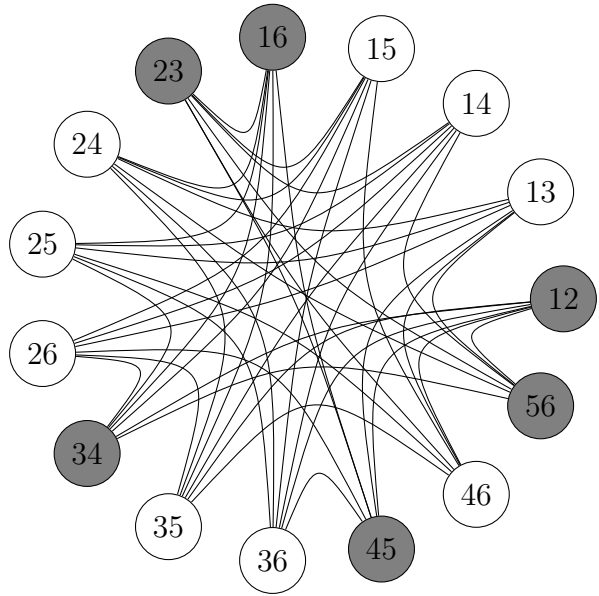

The metric dimension of Kneser graphs has been studied previously \cite{bcggmmp,bc}, and upper bounds were given for a variety of parameter sets.
When $k$ is fixed and $n \geq 3k$, the best asymptotic result of these used a partitioning technique, which yielded an upper bound:
\begin{eqnarray*}
\beta(K(k,n)) &\le & \left\lceil \frac{n}{2k-1} \right\rceil   \left( \binom{2k-1}{k}-1  \right) \hspace{5pt} \\
&\sim& \hspace{5pt} n \frac{2^{2k}}{k\sqrt{\pi k}}.
\end{eqnarray*}
We provide an upper and lower bound that differs from $n/2 + n/k$ by an additive constant for both the localization number and metric dimension of Kneser graphs in infinitely many cases, and represents an improvement for all cases where $n$ is sufficiently large and $k\geq 3$. Our proofs rely on the new notion of hypergraph detection, which we define next.

\subsection{Hypergraph detection}

A \emph{hypergraph} $H=(V,E)$ is a set of vertices $V$ along with a collection $E$ of subsets of $V$.
The elements of $E$ are called \emph{hyperedges}.
We write $V(H)$ and $E(H)$ to represent the vertex set and hyperedge set associated with the hypergraph $H$.
If all hyperedges have cardinality $k$, we say that the hypergraph is \emph{$k$-uniform}.
If each vertex in $V$ is contained in exactly $r$ hyperedges, we say that the hypergraph is \emph{$r$-regular}.
We consider the notions of a cycle and girth originally due to Berge \cite{berge}, in which a cycle of length $\ell$ is defined as a sequence of $\ell$ vertices $(v_1, v_2, \ldots, v_\ell)$ such that $\{v_i,v_{i+1}\}$ occur in some hyperedge for all $1 \leq i \leq  \ell-1$, as do $\{v_1,v_\ell\}$.
The \emph{girth} is the length of the smallest cycle in the hypergraph.
The degree of a vertex $u$, written $d(u)$, is the number of hyperedges that contain $u$.
The neighborhood of $u$, written $N(u)$, is the set of vertices $\{v \mid \{u,v\} \in h \text{ for some } h\in E(H)\}$.

Consider the following hypergraph detection problem.
Let $H=(V,E)$ be a hypergraph with $n$ vertices and $g$ hyperedges with each hyperedge of maximum cardinality $k$.
For any set of $k'$ vertices $B$, where $k' \leq k$, define the \emph{detection vector} $(p_1, p_2, \ldots, p_g)$ by setting $p_i$ to be $0$, $1$, or $k$ if the hyperedge $h_i$ contain zero, at least one (but not $k$), or $k$ vertices in $B$, respectively.
If any two such selections of $B$ will always produce a different detection vector on $H$, then we call the hypergraph \emph{$k'$-detectable}.
We note that if $k' < k$, then $p_i$ will only be $0$ or $1$ for each $i\in[g]$.

When $n \geq 3k$, constructing a $k$-detectable, $k$-uniform hypergraph on $n$ vertices is equivalent to constructing a resolving set for $K(k,n)$.
This is shown from the following two lemmas.
The reason for using this conversion to hypergraphs is twofold: first, it simplifies the discussion significantly, and second, it allows us to apply theorems that were written for hypergraphs more readily.
We first construct a resolving set for $K(k,n)$ from a $k$-detectable, $k$-uniform hypergraph on $n$ vertices.

\begin{lemma} \label{detect2res}
Let $n \geq 3k$.  If $H$ is a $k$-detectable, $k$-uniform hypergraph on $n$ vertices and $g$ hyperedges, then there exists a resolving set for $K(k,n)$ of cardinality $g$.
\end{lemma}
\begin{proof}
We write the hyperedges of $H$ as $E(H) = \{h_1, h_2, \ldots, h_g\}$.
Each hyperedge of $H$ is a $k$-tuple of $[n]$.
We then take $s_i\in V(K(k,n))$ with $s_i = h_i$, for each $i \in [g]$.
We claim that $S = \{s_1, s_2, \ldots, s_g\}$ is a resolving set of $K(k,n)$.
Note that the only difference between $s_i$ and $h_i$ is that we consider $s_i$ to be a vertex of $K(k,n)$ and $h_i$ to be a hyperedge of $H$.

For the sake of contradiction, assume that $S$ is not a resolving set of $K(k,n)$.
There exists vertices $v_1,v_2 \in V(K(k,n))$ such that $d(s,v_1)=d(s,v_2)$ for all $s \in S$.
If $d(s_i,v_1)=0$ for any $s_i \in S$, then $v_1$ and $v_2$ can be distinguished, so we assume $d(s_i,v_1)>0$ and similarly $d(s_i,v_2)>0$.
Consider the two selections of $k$ vertices, $B_1=v_1$ and $B_2=v_2$, in the hypergraph detection problem on $H$.
Note that the only difference between $B_i$ and $v_i$ is that the first is a selection of $k$ vertices in a hypergraph, and the second is a vertex of $K(k,n)$.
We will show that $B_1$ and $B_2$ cannot be distinguished using their detection vectors, and so we will have the required contradiction.

For each $s_i \in S$, if $d(s_i,v_1)=d(s_i,v_2)=1$, then $s_i \cap v_1 = s_i \cap v_2 = \emptyset$.
This implies that the hyperedge $h_i$ does not contain any vertices in common with $B_1$ and $B_2$, and so $p_i$ of the detection vector will return $0$ for both $B_1$ and $B_2$.

For each $s_i \in S$, if $d(s_i,v_1)=d(s_i,v_2)=2$, then $s_i \cap v_1 \neq \emptyset$ and $s_i \cap v_2 \neq \emptyset$. Also note that both $s_i \cap v_1$ and $s_i \cap v_2$ have cardinality strictly less than $k$.
This implies that the hyperedge $h_i$ does contain some vertex (but not $k$ vertices) in common with $B_1$ and $B_2$, and so $p_i=1$ for both $B_1$ and $B_2$.

In both cases, each $p_i$ is the same in both of the detection vectors for $B_1$ and for $B_2$. However, $B_1 \neq B_2$, so the hypergraph is not $k$-detectable, which forms the contradiction.
Therefore, $S$ is a resolving set of $K(k,n)$ of cardinality $g$, and we are done.
\end{proof}

Likewise, we may construct a $k$-detectable, $k$-uniform hypergraph on $n$ vertices from a resolving set for $K(k,n)$.

\begin{lemma} \label{res2detect}
Let $n \geq 3k$. If there exists a resolving set $S$ for $K(k,n)$ of cardinality $g$, then there exists a $k$-detectable, $k$-uniform hypergraph on $n$ vertices and $g$ hyperedges.
\end{lemma}
\begin{proof}
We write $S = \{s_1, s_2, \ldots, s_g\}$.
Each element of the resolving set $S$ is a $k$-tuple of $[n]$.
We define a $k$-uniform hypergraph $H$ on vertices $[n]$ by defining the edges of $H$ as $h_i=s_i$ for $i \in [g]$.
We claim that $H$ is $k$-detectable.
The only difference between $s_i$ and $h_i$ is that we consider $s_i$ to be a vertex of $K(k,n)$ and $h_i$ to be a hyperedge of $H$.

For the sake of contradiction, assume that $H$ is not $k$-detectable.
As a result, there exist two sets of vertices $B_1$ and $B_2$ of cardinality $k$ on the hypergraph that cannot be distinguished by a detection vector.
This implies that for each $p_i$ in the detection vector, $p_i=0$ for both $B_1$ and $B_2$, or $p_i=1$ for both $B_1$ and $B_2$. We cannot have $p_i=k$, or else the detection vector instantly distinguishes $B_1$ and $B_2$.
We let $v_1=B_1$ and $v_2=B_2$ be two vertices of $K(k,n)$.
To complete the contradiction, we will show that $v_1$ and $v_2$ are not resolved by $S$.

If $p_i=0$ on both $B_1$ and $B_2$, then $h_i$ has no elements in common with $B_1$ or $B_2$, and as a result $s_i$ has no elements in common with $v_1$ or $v_2$, and so $s_i$ has distance $1$ to $v_1$ and to $v_2$.
If $p_i=1$ on both $B_1$ and $B_2$, then $h_i$ has some elements (but less than $k$ elements) in common with $B_1$ or $B_2$, and as a result $s_i$ has some elements (but less than $k$ elements) in common with $v_1$ or $v_2$, and so $s_i$ has distance $2$ to $v_1$ and to $v_2$.

But then every element of $S$ has the same distance to both $v_1$ and $v_2$ in $K(k,n)$, and so $v_1$ and $v_2$ cannot be resolved by $S$, which forms the contradiction.
Therefore, $H$ is a $k$-detectable, $k$-uniform hypergraph on $n$ vertices and $g$ hyperedges, as required.
\end{proof}

As an example, we take the $6$-cycle considered as a $2$-uniform, $2$-regular hypergraph with vertices $\{ 1,2,3,4,5,6\}$ and edge set $E= \{\{1,2\},\{2,3\},\{3,4\}, \{4,5\}, \{5,6\}, \{1,6\}\}$. The reader may verify that this hypergraph is $2$-detectable.
This edge set $E$ also forms a resolving set of $K(2,6)$, which is the resolving set provided in Figure~\ref{ex:Kne}.

In the following, we provide a lower bound on the number of hyperedges in a $k$-detectable hypergraph with hyperedges of cardinality at most $k$.

\begin{lemma} \label{lm:KnesMajor}
Let $k \geq 3$, $n \geq 3k$, and if $k=3$, then let $n\geq 18$.
If $H$ is a $k$-detectable hypergraph of $n$ vertices with each hyperedge having cardinality at most $k$ such that no hyperedge occurs twice, then $H$ has at least $n/2 +n/k$ hyperedges if $k \neq 4,$ and at least $(3n-1)/4$ if $k=4$.
\end{lemma}
\begin{proof}
We assume first that $H$ is a $k$-detectable hypergraph of $n$ vertices with no isolated hyperedges such that $H$ has the smallest number of hyperedges possible.
We will deal with the case that $H$ contains isolated hyperedges at the end of the proof. We first prove two properties.

\medskip
\noindent \emph{Property 1:}
If $u,v\in V(H)$ such that $v \notin N(u)$, then $d(u)+d(v) \geq k$.
\medskip

To prove the property, assume for the sake of contradiction that  $d(u)+d(v) < k$.
This will imply that there are two $k$-sets $B_1$ and $B_2$ of vertices that are indistinguishable, contradicting $H$ being $k$-detectable.
We may assume that $d(u)<d(v)$ without loss of generality.
As a result, if $d(u) \geq k/2$ we are done, so we assume $d(u) \leq \lfloor k/2 \rfloor$.
Note that by assumption, $\{u,v\}$ does not occur in a hyperedge of $H$.

We select a set of $k-1$ vertices $T \subseteq V(H) \setminus \{u,v\}$ that contains at least one vertex from each hyperedge in $I_u$ and $I_v$, where $I_u$ is the set of hyperedges incident to vertex $u$.
This is always possible, as first we have assumed $d(u)+d(v) < k$ so $T$ has enough vertices to cover each hyperedge incident with $u$ and $v$, and second that there are at least $k-1$ vertices in $V(H) \setminus \{u,v\}$ since $n-2 \geq 3k-2 \geq k-1$, meaning that there are enough unique vertices to define $T$.
Note that while $T$ contains one vertex from each edge incident to $u$ and $v$, $T$ may also contain other vertices of $V(H)\setminus \{u,v\}$, even if they are not incident to $u$ or $v$.

We also select $T$ such that there is no hyperedge $h\in E(H)$ of cardinality $k$ such that $h \subseteq T\cup\{u,v\}$.
To see this is possible, suppose we have chosen $T$ containing some hyperedge $h_1\in E(H)$ of cardinality $k$ such that $h_1 \subseteq T\cup\{u,v\}$.
If we remove any vertex from $T$, then $h_1$ will still intersect $T$ as $h_1$ intersects $T$ in $k-1$ vertices.
Also, there are $k-2$ hyperedges in $(I_u\cup I_v) \setminus \{h_1\}$ that intersect $T$.
Therefore, some vertex $x_1$ can be removed from $T$ such that $T\setminus\{x_1\}$ intersects all hyperedges in $I_u\cup I_v$.
We can then add a vertex $x_2\in V(H) \setminus \{u,v,x_1\}$ to $T \setminus \{x_1\}$ to form $T'$.
If some hyperedge $h_2\in E(H)\setminus \{h_1\}$ of cardinality $k$ has $h_2 \subseteq T'\cup\{u,v\}$, then we can instead add some other vertex $x_3\in V(H) \setminus \{u,v,x_1,x_2\}$, and so on.
There can be at most $k-1$ points $\{x_1, x_2, \ldots, x_{k-1}\}$ constructed in this way, as there are at most $k-1$ hyperedges in $I_u \cup I_v$ by our assumption.
As $n\geq 3k$, there are $n-2k \geq k$ vertices in $V(H) \setminus \{u,v\}$ that are not in $T$ or in $\{x_1, x_2, \ldots, x_{k-1}\}$.
Adding one of these vertices to $T\setminus \{x_1\}$ then results in the set of vertices required.
As a result, we can assume that $T$ was chosen such that there is no hyperedge $h\in E(H)$ of cardinality $k$ such that $h \subseteq T\cup\{u,v\}$.

We may choose either the $k$ vertices in $B_1 = T \cup \{u\}$ or the $k$ vertices in $B_2 = T \cup \{v\}$. In both $B_1$ and $B_2$, the value of $p_i$ in the detection vector will be the same for each $i$, in that all hyperedges $h_i$ that contain a vertex in $T\cup\{u,v\}$ will yield $p_i=1$ (no $p_i=k$ due to the construction of $T$) and that all hyperedges $h_i$ that do not contain a vertex in $T\cup\{u,v\}$ will yield $p_i=0$, independent of whether $u$ or $v$ is added to $T$.
This contradicts the assumption that $H$ is $k$-detectable.
Thus, the claim that $d(u)+d(v) \geq k$ must be true and Property 1 holds.

\medskip
\noindent \emph{Property 2:}
If $u,v\in V(H)$ with $v \in N(u)$, then $d(u)+d(v) \geq k+2$.
\medskip

To prove Property 2, assume for the sake of contradiction that $d(u)+d(v) < k+2$.
Say that $\{u,v\}$ occurs in some hyperedge $h$.
There may be many such hyperedges, and we can pick any one except when $\{u,v\}$ is a hyperedge in $E(H)$, in which case we must have $h=\{u,v\}$.

Select a set of $k-1$ vertices $T \subseteq V(H)\setminus \{u,v\}$ that contains at least one vertex from each hyperedge in $I_u\setminus \{h\}$ and in $I_v \setminus\{h\}$.
This is always possible for two reasons.
First, because we have assumed $d(u)-1+d(v)-1 < k,$ and so there are enough vertices in $T$ to intersect each of these hyperedges.
Second, there are at least $n-k > k-1$ vertices $V(H)$ that are not contained in $h$; thus, there are a sufficient number of vertices in $V(H)\setminus h$ from which to choose $T.$

Let $E'$ be the set of hyperedges that intersect $u$ or $v$, but not both.
Note that $|E'| \leq d(u)-1 +d(v)-1 \leq k-1.$
We also select $T$ such that there is no hyperedge $h'\in E'$ of cardinality $k$ such that $h' \subseteq T\cup\{u,v\}$.
To see this is possible, suppose we have chosen $T$ containing hyperedge $h_1\in E'$ of cardinality $k$ such that $h_1 \subseteq T\cup\{u,v\}$.
If we remove any vertex from $T$, then $h_1$ will still intersect $T$ as $h_1$ intersects $T$ in $k-1$ vertices.
Also, there are at most $k-2$ hyperedges in $(I_u\cup I_v) \setminus \{h,h_1\}$, each of which must intersect $T$.
Therefore, some vertex $x_1$ can be removed from $T$ such that $T\setminus\{x_1\}$ intersects all hyperedges in $(I_u\cup I_v)\setminus h$.
We can then add a vertex $x_2\in V(H) \setminus \{u,v,x_1\}$ to $T \setminus \{x_1\}$ to form $T'$.
If some hyperedge $h_2\in E'\setminus \{h_1\}$ of cardinality $k$ has $h_2 \subseteq T'\cup\{u,v\}$, then we can instead add some other vertex  $x_3\in V(H) \setminus \{u,v,x_1,x_2\}$, and so on. There can be at most $|E'|$ points $\{x_1, x_2, \ldots, x_{|E'|}\}$ constructed in this way, as there are at most $|E'|$ hyperedges in $E'$.
As $n\geq 3k$ and $|E'|\leq k-1$, there are $n-2-(k-1)-|E'| \geq n-2k \geq k$ vertex in $V(H) \setminus \{u,v\}$ that are not in $T$ or in $\{x_1, x_2, \ldots, x_{|E'|}\}$.
Adding one of these vertices to $T\setminus \{x_1\}$ then results in the set of vertices required.
As a consequence, we can assume that $T$ was originally chosen such that there is no hyperedge $h\in E'$ of cardinality $k$ such that $h \subseteq T\cup\{u,v\}$.

Let $B_1=T \cup \{u\}$ and $B_2=T \cup \{v\}$.
In either case, each $p_i$ is identical in the detection vectors for both $B_1$ and $B_2$, in that all hyperedges $h_i$ that contain a vertex in $T\cup\{u,v\}$ yield $p_i=1$ ($p_i\neq k$ due to the construction of $T$) and that all hyperedges $h_i$ that do not contain a vertex in $T\cup\{u,v\}$ yield $p_i=0$, independent of whether $u$ or $v$ is added to $T$.
This contradicts the assumption that $H$ is $k$-detectable.
Thus, the claim that $d(u)+d(v) \geq k+2$ must be true and Property 2 holds.
\smallskip

Let $u$ be a vertex of minimum degree $\delta$ and let $S$ be the set of vertices of degree $k/2$ or $(k+1)/2$. We consider five cases, and in each, show that the number of hyperedges is at least $n/2+n/k$ if $k \neq 4$ and at least $(3n-1)/4$ if $k=4$.
For ease of notation, we define $m=n/2+n/k$, and so $(3n-1)/4 = m - 1/4.$

\medskip
\noindent \emph{Case 1:} $\delta < k/2-1$. In this case, there is one vertex that obtains the minimum degree of at most $k/2-3/2$, and the remaining $n-1$ vertices have degree at least $k/2 + 3/2$, by Properties 1 and 2. The number of hyperedges can be calculated as
\begin{align*}
\sum_{u\in V(H)} \frac{d(u)}{k}   &\geq 	\frac{\delta + (n-1)(k/2+3/2)}{k} \\
&\geq (n-1)\left(\frac{1}{2}+\frac{3}{2k}\right)\\
&= \frac{n}{2}+\frac{n}{k} + \frac{n}{2k} - \frac{1}{2} - \frac{3}{2k}\\
&\geq m.
\end{align*}
Note that the inequality holds for all $k\ge 3.$
\medskip

\noindent \emph{Case 2:} $\delta = k/2-1$. We then have that $S$ is empty by Properties~1 and 2 as all vertices $V(H)\setminus\{u\}$ will have degree at least $k/2+1$.
Let $u$ be the unique vertex of degree $\delta$, and let $S'$ denote the vertices of degree exactly $k/2+1$. Note that the vertices in $S'$ cannot be neighbors of $u$, by Property~$2$.
As there are no repeated hyperedges, there must be at least $\lfloor \log_2(\delta)+1 \rfloor$ vertices in the hyperedges from $u$.
We say that $S''$ is the set of vertices of degree at least $k/2 +2$, so $|S''|\geq \lfloor \log_2(\delta)+1 \rfloor$.
We have that $|S'| = n-1-|S''|$. The number of hyperedges may be calculated as
\begin{align*}
\sum_{u\in V(H)} \frac{d(u)}{k}
&\geq 	\frac{(k/2-1) +|S'|(k/2+1) + |S''|(k/2+2)}{k} \\
&= \frac{n}{2}+\frac{n}{k} - \frac{2}{k} + \frac{|S''|}{k}  \\
&\geq \frac{n}{2} + \frac{n}{k} + \frac{\lfloor \log_2(\delta)\rfloor-1}{k}.
\end{align*}
When $k \geq 6$, this value is greater than or equal to $m$. Otherwise, $k=4$ since $\delta=k/2-1$ implies that $k$ is even, and the number of hyperedges is greater than or equal to $m-1/4$.

\medskip
\noindent \emph{Case 3:} $\delta = (k-1)/2$. In this case, $S$ is potentially non-empty. We investigate the cardinality of $S$. By Property~2, any pair of vertices in $S\cup \{u\}$ cannot occur together in a hyperedge.
There are at least $(k+1)/2$ hyperedges incident with each vertex in $S$ (as $k$ is odd in this case) and at least $(k-1)/2$ hyperedges incident with $u$.
Each of these hyperedges are unique, giving $|S|(k+1)/2  + (k-1)/2$ hyperedges that intersect $S\cup \{u\}$.
We assume $|S| \leq n/(k+1) + 2n/(k(k+1))-(k-1)/(k+1)$,  %Can improve here
or else the number of hyperedges that intersect $S\cup \{u\}$ is larger than $n/2 + n/k$, %\trent{here}
and so we would be finished.
As each vertex not in $S \cup \{u\}$ has degree at least $(k+1)/2+1$, we have the number of hyperedges as
\begin{align*}
\sum_{u\in V(H)} \frac{d(u)}{k}
&\geq \frac{\delta + |S|(k+1)/2 + (n-|S| -1)((k+1)/2+1)}{k} \\
&= \frac{n}{2} + \frac{n}{k}+ \frac{n}{2k} -\frac{2}{k}-\frac{|S|}{k} \\
&\geq \frac{n}{2} + \frac{n}{k} +\frac{n}{2k}-\frac{2}{k} - \left(\frac{n}{k(k+1)} + \frac{2n}{k^2(k+1)}-\frac{k-1}{k(k+1)}\right).
\end{align*}
By direct calculation, it may be verified that this last value is greater than or equal to $m$ when $n \geq (2k^2 + 6k)/(k^2-k-4)$, which is always true for $k\geq 4$ and $n\geq 3k$, and is true for $k=3$ when $n \geq 18$.

\medskip
\noindent \emph{Case 4:} $\delta = k/2$.
We then have that $S$ is non-empty.
In this case, we also investigate the set $S'$ of vertices of degree $k/2+1$.
The number of hyperedges in $H$ is then
\begin{align*}
\sum_{u\in V(H)} \frac{d(u)}{k}  &\geq \frac{k}{2} |S| + \left(\frac{k}{2}+1\right)|S'| + \left(\frac{k}{2}+2\right)(n-|S|-|S'|) \\
&= \frac{n}{2} + \frac{n}{k}+ \frac{n}{k} - \frac{2|S|}{k} - \frac{|S'|}{k}.
\end{align*}
We are done, unless $n/k - 2|S|/k - |S'|/k<0$.
This can only be the case if $|S'| > n-2|S|$.
As such, we suppose that this is the case, and we will show that the number of hyperedges is larger than $m$.
We note that any hyperedge containing a vertex in $S$ cannot contain a vertex in $S$ or $S'$ by Property~2.
There are $k|S|/2$ hyperedges that contain a vertex in $S$.
There are at least a further $(k/2+1)|S'|/k = (1/2+1/k)|S'|$ hyperedges that contain a vertex in $S'$.
This implies that there are at least
\begin{align*}\frac{k}{2} |S| + \left(\frac{1}{2}+\frac{1}{k}\right)|S'| &> \frac{k}{2} |S| + \left(\frac{1}{2}+\frac{1}{k}\right)(n-2|S|) \\
&= \frac{n}{2}+\frac{n}{k}  +|S|\left(\frac{k}{2}-2\left(\frac{1}{2}+\frac{1}{k}\right)\right)\\
&\geq m
\end{align*} hyperedges in $H$, where this last inequality holds as $\delta=k/2$ implies $k$ is even and so $k\geq 4$ in this case.

\medskip
\noindent \emph{Case 5:} $\delta = (k+1)/2$. We then have that $S$  is non-empty. We investigate the cardinality of $S$. By Property~2, any pair of vertices in $S$ cannot occur together in a hyperedge.
There are at least $(k+1)/2$ hyperedges incident with each vertex in $S$. Each of these hyperedges are unique, giving $(k+1)|S|/2$ hyperedges that intersect $S$.
We assume $|S| \leq n/(k+1) + 2n/(k(k+1))$,  %Can improve here
or else the number of hyperedges that intersect $S$ is larger than $n/2 + n/k$,
and so we would be finished.
As each vertex not in $S$ has degree at least $(k+3)/2$, we have the number of hyperedges as
\begin{align*}
\sum_{u\in V(H)} \frac{d(u)}{k}
&\geq \frac{|S|(k+1)/2 + (n-|S|)(k+3)/2}{k} \\
&= 	\frac{n}{2} + \frac{n}{k} + \frac{n}{2k} -\frac{|S|}{k} \\
&\geq \frac{n}{2} + \frac{n}{k} + \frac{n}{2k} -\frac{n}{k(k+1)}-\frac{2n}{k^2(k+1)}\\
&\geq  m,
\end{align*}
where the last inequality follows as $k \geq 3$.
\smallskip

Hence, in all cases with no isolated hyperedges, we have that the hypergraph has the desired number of hyperedges.  Now consider the final case, where $H$ contains an isolated hyperedge, say $\{u\}$.
Form the hypergraph $H'$ by removing $u$ from all hyperedges in $H$ except for hyperedge $\{u\}$.
Note that $H$ and $H'$ contain the same number of vertices and hyperedges.
If $H$ is $k$-detectable, then so is $H'$.
However, if $H'$ is $k$-detectable, $H$ may not be.
It is then always beneficial to restrict to hypergraphs where any vertex that is contained in an isolated hyperedge will have degree $1$.
Suppose $H'$ contains $p$ isolated hyperedges.
We construct a hypergraph $H''$ by removing these $p$ isolated hyperedges from $H'$.
Observe that $H''$ has $n-p$ vertices and no isolated hyperedges, and as our work above shows, $H''$ has at least $(n-p)/2 +(n-p)/k$ hyperedges if $k \neq 4$ and at least $(n-p)/2 + (n-p)/4-1/4$ if $k=4$.
We can then conclude that   $H'$ has at least
\[\frac{n-p}{2} +\frac{n-p}{k} + p \geq \frac{n}{2} +\frac{n}{k}\] hyperedges if $k \neq 4$ and at least $(n-p)/2 +(n-p)/4-1/4+p \geq n/2 +n/4-1/4$ if $k=4$.
As a result, the proof is complete.
\end{proof}

An immediate consequence of Lemmas~\ref{detect2res}, \ref{res2detect}, and \ref{lm:KnesMajor} is the following.
\begin{corollary}\label{cordimd}
If $n \geq 3k$ and $k \geq 3$, then $\beta(K(k,n)) \geq n/2+ n/k.$ If $k = 4$, then $\beta(K(4,n)) \geq (3n-1)/4$.
\end{corollary}

Perhaps surprisingly, we may use Lemma~\ref{lm:KnesMajor} to provide a lower bound on the localization number of Kneser graphs.

\begin{lemma}\label{lemmlocc}
If $n \geq 3k$ and $k \geq 3$, then
\begin{equation*}
\zeta(K(k,n)) \geq \frac{n}{2}+\frac{n}{k}-\frac{k}{2} -1.
\end{equation*}
 If $k = 4$, then
\begin{equation*}\zeta(K(4,n)) \geq \frac{3n-13}{4}.
\end{equation*}
\end{lemma}
\begin{proof}
If we have $g=\zeta(K(k,n))$ cops, then there is some time during play when the robber moves from a vertex $u$ to a vertex in $N(u)$, and is then captured by the cops.
An observation to make here is that $N(u)$ contains all the vertices of $K(k,n)$ that contain no elements of $u$. (Recall that $u$ is a $k$-tuple of $[n]$.)
As a result, the vertices in $N(u)$ will be labeled by each $k$-tuple of $[n] \setminus u$, and the induced subgraph of $K(k,n)$ on the vertex set $N(u)$, which we write as $K(k,n) \mid_{N(u)}$ for the rest of the proof, is isomorphic to $K(k,n-k)$.

Suppose that during the final play, the cops were on vertices $X=\{s_1, s_2, \ldots, s_g\}$.
Define the hyperedges $h_i = s_i \setminus u$, and consider the set $X' = \{h_1, h_2, \ldots, h_g\}$.
Observe that $X'$ can be considered as a set of hyperedges with cardinality at most $k$ from a hypergraph $H$ on the $n-k$ vertices $V(K(k,n)) \setminus u$.
We remove any repeated hyperedges from $H$ and adjust $g$ accordingly.
We now show that $H$ is $k$-detectable.

For the sake of contradiction, assume that $H$ is not $k$-detectable.
There exists two $k$-sets $B_1$ and $B_2$ of vertices (selected from $V(H) = V(K(k,n)) \setminus u$) on the hypergraph with identical detection vectors.
This implies that each hyperedge $h_i$ will yield $p_i=0$ for both $B_1$ and $B_2$, or $p_i=1$ for both $B_1$ and $B_2$.
Note that a hyperedge $h_i$ will not yield $p_i=k$, or else $B_1$ and $B_2$ are instantly distinguishable.
We let $v_1=B_1$ and $v_2=B_2$ be two vertices of $K(k,n)\mid_{N(u)}$.
To complete the contradiction, we will show that $v_1$ and $v_2$ are not resolved by $S$.

If hyperedge $h_i$ yields $p_i=0$ for both $B_1$ and $B_2$, then $h_i$ has no elements in common with $B_1$ or $B_2$, and as a result $s_i$ has no elements in common with $v_1$ or $v_2$, and so $s_i$ has distance $1$ to $v_1$ and to $v_2$.
If hyperedge $h_i$ yields $p_i=1$ for both $B_1$ and $B_2$, then $h_i$ has some elements (but not $k$ elements) in common with $B_1$ and with $B_2$, and as a result $s_i$ has some elements (but not $k$ elements) in common with $v_1$ and $v_2$, and so $s_i$ has distance $2$ to $v_1$ and to $v_2$.

But then every element of $X$ has the same distance to $v_1$ and $v_2$, and so $v_1$ and $v_2$ cannot be resolved by $S$, which forms the contradiction.
Therefore, we have that $H$ must be $k$-detectable.
We then have that $H$ is a $k$-detectable hypergraph with each hyperedge having cardinality at most $k$ such that no hyperedge occurs twice.
The result then follows from  Lemma~\ref{lm:KnesMajor} by substituting $n-k$ for $n.$
\end{proof}

This completes our discussion of the lower bounds for both the metric dimension and localization number for Kneser graphs. Next, we consider an upper bound of the metric dimension for the Kneser graphs. Upper bounds are given on the minimum number of hyperedges required in a $k'$-detectable, $k$-uniform hypergraph, obtained by showing that any $k$-uniform hypergraph with a sufficiently large minimum degree and girth implies the hypergraph is $k'$-detectable. We then utilize an example of a $k$-uniform hypergraph with a sufficiently large minimum degree and girth.

\begin{lemma}
If $H$ is a $k$-uniform hypergraph with minimum degree at least $k'/2+1$ and girth at least $5$, where $k' \leq k$, then $H$ is $k'$-detectable.
\end{lemma}
\begin{proof}
Suppose that $B$ is a selection of a set of $k'$ vertices.
Let  $E'$ be the subset of hyperedges in $E=H(E)$ containing a vertex of $B$.
That is, $E'$ contains those hyperedges $h_i$ that yield $p_i=1$.
Note that if $v \in B$, then every hyperedge in $E$ that intersects $v$ must be in $E'$.
Said in another way, if some vertex $v$ occurs in less than $d(v)$ hyperedges of $E'$, then it implies that $v \notin B$.
Define $S$ as the collection of vertices $v$ such that every hyperedge in $E$ that intersects $v$ is in $E'$.

We note that $B \subseteq S$. However, we assert that either $S = B$ or $S=B \cup \{v_1\}$ for some vertex $v_1$.
To see this, suppose for the sake of contradiction that a second vertex $v_2\notin B \cup \{v_1\}$ is in $S$.
Each hyperedge adjacent to $v_1$ must contain a vertex of $B$, and so at least $k'/2+1$ vertices in $B$ have distance $1$ from $v_1$.
Similarly,  at least $k'/2+1$ vertices in $B$ have distance $1$ from $v_2$.
As the hypergraph's girth is at least $5$, there can only be one vertex of $B$ that is distance $1$ from both $v_1$ and $v_2$.
But then there are at least $k'+1$ distinct vertices in $B$ of distance $1$ from either $v_1$ or $v_2$, which gives the contradiction.

Now we must show that given $S$, the cop player may deduce $B$.
If $S$ has cardinality $k'$, then the cop player can immediately deduce that $B=S$.
Otherwise, suppose that $S$ has cardinality $k'+1$.
There is precisely one vertex $v$ in $S$ such that every hyperedge that contains $v$ also contains another vertex of $S$.
To see that there is at least one such vertex, note that the one vertex in $S \setminus B$ must have this property.
To see that there cannot be two such vertices, suppose for the sake of contradiction that there is another vertex $v'\in S$ with this property.
As the girth of the hypergraph is at least $5$, there can only be one vertex of $B$ that is distance $1$ from both $v$ and $v'$.
But then there are at least $k'+1$ vertices in $B$ of distance $1$ from either $v$ or $v'$, which gives the contradiction.
As a result, the cops identify that $B=S\setminus\{v\}$.
\end{proof}

Note that if the minimum degree is large, then the hypergraph is $k'$-detectable for many values of $k'$.

\begin{corollary} \label{cor:all_num_cops}
If $H$ is a $k$-uniform hypergraph with minimum degree at least $k/2+1$ and girth at least $5$, then $H$ is $k'$-detectable for all $k' \leq k$.
\end{corollary}

In the following lemma, we provide a hypergraph with the properties required by the antecedent of Corollary \ref{cor:all_num_cops}, and hence we can construct a resolving set of $K(k,n)$, giving an upper bound on the metric dimension of the Kneser graphs.
The existence of a $k$-uniform, $\lceil k/2+1 \rceil$-regular hypergraph with girth $5$ was shown in \cite{el}.
We label such a hypergraph as $H(k,5)$.
This construction was probabilistic, and so the number of vertices is not explicitly known.
Let $m$ be the number of vertices in $H(k,5)$.
With $w=k^{\lceil k/2+1 \rceil}$, we have that $m$ is bounded above by $w  2^{w(w-1)/2}$.

\begin{lemma}\label{lemdimd}
For $k\geq 4$ a fixed even integer,
\begin{align*}
\beta(K(k,n)) &\leq \left(\frac{1}{2}+\frac{1}{k}\right)n+ \left(\frac{1}{2}+\frac{1}{k}\right)m \left \lceil \frac{n'}{m} \right \rceil \\&= \left(\frac{1}{2}+\frac{1}{k}\right)n +O(1),\end{align*}
and for $k \geq 3$ a fixed odd integer,
\begin{align*} \beta(K(k,n)) &\leq \left(\frac{1}{2}+\frac{1}{k}\right)n+ \left(\frac{1}{2k}\right)n +\left(\frac{1}{2}+\frac{1}{k}+\frac{1}{2k}\right)m \left  \lceil \frac{n'}{m} \right \rceil\\
&= \left(\frac{1}{2}+\frac{1}{k}+\frac{1}{2k}\right)n +O(1),\end{align*}
where $n'$ is the smallest non-negative integer $n'\equiv n \pmod{m}$, and $m$ is the number of vertices in $H(k,5)$.
\end{lemma}
\begin{proof}
For any set of $m$ vertices $V$, we may place a $H(k,5)$, and we label such a hypergraph as $H_V$.
Consider the partition of $V(H)=[n]$ into parts $P_1, P_2, \ldots, P_r$ of cardinality $m$, where we allow the last part to cover some elements of a previous part (so strictly speaking, this is a cover).
Define the hypergraph $H = \bigcup_{i=1}^r H_{P_i}$.
The subset of vertices $V(H_{P_i})$ along with corresponding hyperedges $E(H_{P_i})$ form a $k$-uniform hypergraph with minimum degree $\lceil k/2+1 \rceil$, and due to Corollary \ref{cor:all_num_cops}, this hypergraph is $k'$-detectable for all $k' \leq k$, implying that the set $B\cap P_i$ can be accurately detected.
As a result, $B$ can be accurately detected, since $B=\bigcup_{i=1}^r (B \cap P_i)$.
As a consequence, $H$ can be used to form a resolving set of $K(k,n)$ by Lemma~\ref{detect2res}.
By calculating the number of hyperedges in $H$ as $\sum_v d(v) /k = rm\lceil k/2+1 \rceil/k$, we have the upper bound given in the statement of the lemma.
\end{proof}

We summarize these results as follows. Note that in the next theorem, the upper bound follows as $\zeta(G) \leq \beta(G)$ for all graphs $G$.

\begin{theorem}\label{finall}
For the localization number and metric dimension of Kneser graphs, we have the following.
\begin{enumerate}
\item For a fixed even integer $k \geq 4$ and $n$ with $n \geq 3k$, we have that
\[\zeta(K(k,n)) =  \frac{n}{2} + \frac{n}{k} +O(1)\]
and \[ \beta(K(k,n)) = \frac{n}{2} + \frac{n}{k} +O(1).\]
\item For a fixed odd integer $k \geq 3$ and $n$ with $n \geq 3k$, we have that
\[\frac{n}{2}+\frac{n}{k} -\frac{k}{2}-1 \leq \zeta(K(k,n)) \leq   \frac{n}{2} + \frac{n}{k}+ \frac{n}{2k} +O(1)\]
and \[ \frac{n}{2}+\frac{n}{k} \leq \beta(K(k,n)) \leq  \frac{n}{2} + \frac{n}{k} + \frac{n}{2k}+O(1).\]
\end{enumerate}
\end{theorem}

We note that for fixed even $k\geq 6$, Corollary~\ref{cordimd} and Lemma~\ref{lemdimd} give that $\beta(K(k,n)) = n/2+n/k$ for an infinite number of values of $n$. We conjecture that this equality holds for all $k \geq 3$ and all $n$ sufficiently large.

\section{Graphs of diameter $2$ with no 4-cycle}\label{secc4}

As referenced in the introduction, there are three subclasses of graphs that have diameter $2$ and contain no $4$-cycle: graphs with maximum degree $n-1$, the Moore graphs, and the polarity graphs~\cite{bef}. In this section, we determine the localization number
of the Moore graphs of diameter $2$ and give bounds on the metric dimension. We conclude by giving bounds on the metric dimension and the localization number of polarity graphs.

\subsection{Moore graphs of diameter $2$} \label{sec:moore}

A \emph{Moore graph} is a graph of diameter $d$ and girth $2d+1$. We give bounds on the localization number for Moore graphs with diameter $2$ that differ by $1$.
Note these are girth $5$ graphs, which are $k$-regular and with $k^2 + 1$ vertices for some positive integer $k$. The known Moore graphs of diameter $2$ are the $5$-cycle, the Petersen graph, and the Hoffman-Singleton graph. The Hoffman-Singleton graph is $7$-regular with order $50$ and $175$ edges; see Figure~\ref{fig: Hoffman-Singleton Graph}.
The only remaining possible Moore graph of diameter $2$ is a hypothetical one that is $57$-regular and order 3,250. For a survey of the theory of Moore graphs, see~\cite{MS}.
\begin{figure}[h!]
\includegraphics[width=5cm]{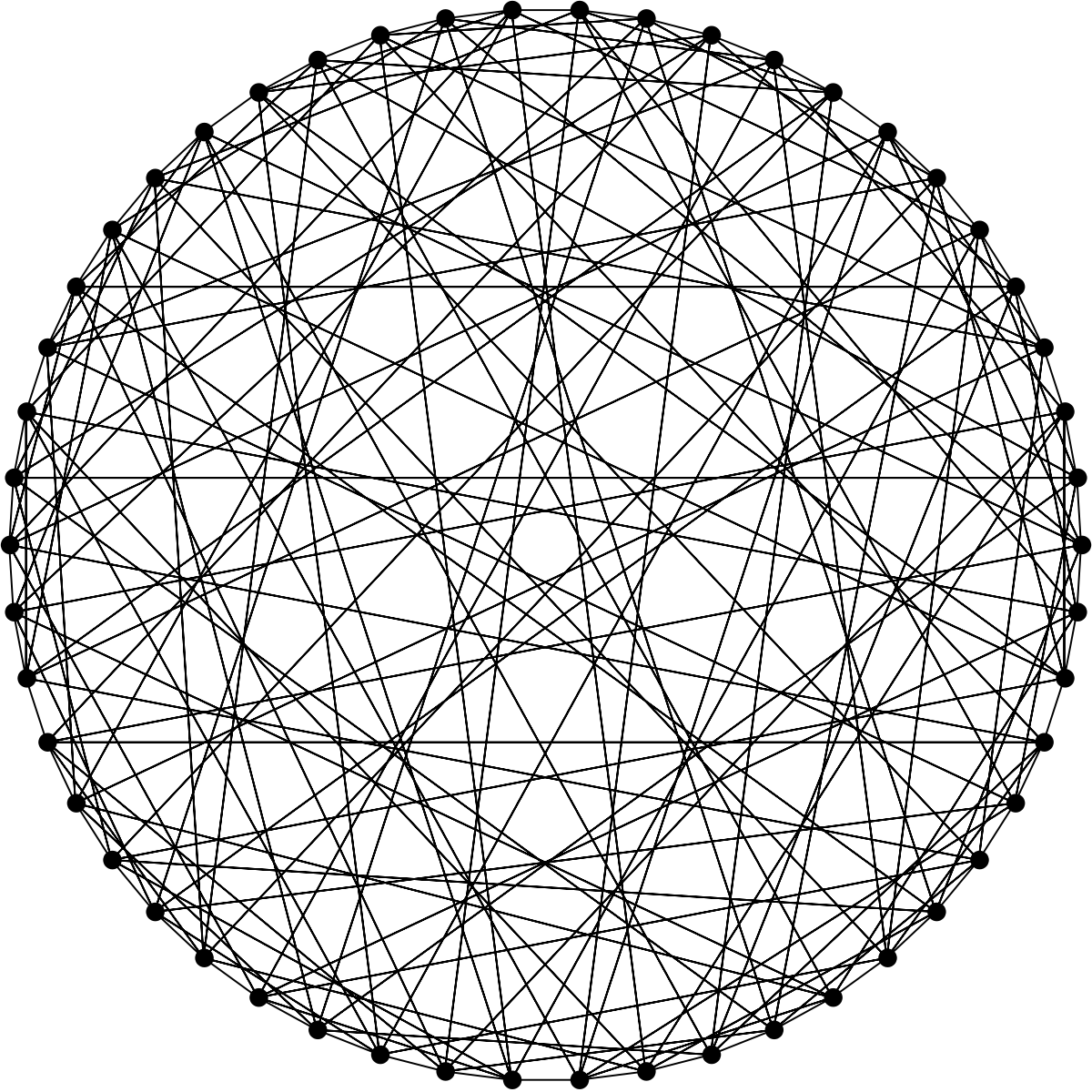}
\caption{The Hoffman-Singleton graph.}\label{fig: Hoffman-Singleton Graph}
\end{figure}

We begin by analyzing the metric dimension of the Moore graphs.
It is straightforward to see that $\beta (C_5) = 2$. We provide bounds for the metric dimension of the larger Moore graphs as follows.

\begin{theorem} \label{thm:MooreMD}
For $k \geq 3$, a Moore graph $G$ of diameter $2$ and girth $5$ that is $k$-regular has
\[
k \leq \beta(G) \leq 2k-3.
\]
\end{theorem}

\begin{proof}
We frame the metric dimension in terms of the localization game where the cops must capture the robber on their first move.
For the lower bound, suppose that some vertex $u$ contains a robber.
All but one of the vertices in $N(u)$ must be distance $0$ or $1$ from any fixed cop, or else there are two vertices in $N(u)$ that have the same distance to all cops.
However, any other cop can have distance $0$ or $1$ to at most one vertex in $N(u)$.
As a result, we need an additional $k-1$ cops, and so require at least $k$ cops in total.

For the upper bound, we play with $2k-3$ cops, and show that the robber can be caught on the first play.
Let $u$ be any vertex in $G$, $v \in N(u)$, and $w \in N(v)\setminus \{u\}$.
The cops place themselves on the $2k-3$ vertices in $(N(u) \cup N(v)) \setminus \{u,v,w\}$.
If the robber is on $u$, then there are at least $2$ cops on vertices in $N(u)$ that probe a distance of $1$, and so the robber is found.
If the robber is on a vertex of $N(u)\setminus \{v\}$, then a cop probes a distance of $0$, and the robber is found.
If the robber is on $w$, then all the cops probe a distance of $2$, and it is straightforward to see that no other vertex has distance $2$ to each cop.
If the robber is on $N(v)\setminus \{u,w\}$, then a cop probes a distance of $0$, and the robber is found.
If the robber is on the second neighborhood of $u$ but not on $N(v)$, then exactly one cop $C_1$ in $N(u)$ probes one.
If in addition to this, there is a cop $C_2$ on $N(v)$ that probes a distance of $1$, then the robber is found on the single vertex in $N(C_1) \cap N(C_2).$ If there is no cop on $N(v)$ that probes a distance of $1$, then the robber is found on $N(C_1) \cap N(w).$
If any cop in $N(v)$ probes a distance of $1$ and none of the cops in $N(u)$ probe a distance of $1$, then the robber is located on $v$.
This is a complete case analysis, that shows that each vertex can be resolved by the cops immediately, so we are done.
\end{proof}

In the case that $k=3$, which is the Petersen graph, this gives the exact value. Note that Theorem~\ref{thm:MooreMD} gives that the Petersen graph has metric dimension $3$.

It is straightforward to see that $\zeta (C_5) = 2$. The localization number of the Petersen graph is $3$, as derived in the next theorem.
\begin{theorem}
The localization number of the Petersen graph is $3$.
\end{theorem}
\begin{proof}
The upper bound follows from Theorem~\ref{thm:MooreMD}, so we focus on the lower bound. We play with two cops, and show that this is insufficient to capture the robber.
Suppose that the cops could find that the robber was on either vertex $x$ or $y$, but did not know which one.
During the robber turn, the robber moves to some vertex in $N[x] \cup N[y]$.
If we suppose that the cops will be able to capture the robber on the next turn, then $x$ and $y$ must be adjacent.
To see this, if we suppose that $x$ and $y$ are not adjacent, then $|N[x] \cup N[y]| = 7$ and there are at most six distinct distance vectors that the cops may use to distinguish them, and so at least two vertices in $N[x] \cup N[y]$ cannot be distinguished using distance vectors.
So suppose that $x$ and $y$ are adjacent.
As $|N[x] \cup N[y]|=6$ and there are six distinct distance vectors, the cops must be placed so that each of the six possible distance vectors occur within the vertices of $N[x] \cup N[y]$.
This means that the cops must be placed on vertices in $N[x] \cup N[y]$, so that two distance vectors containing $0$'s will occur.
It also means that the cops must be distance $2$ from each other, so the distance vector $(1,1)$ occurs.
Up to symmetry, there is only one way to place the cops to have these two properties. However, under this placement, there are two vertices in $N[x] \cup N[y]$ that are indistinguishable.
As such, the robber is not caught during the next turn, and can continue to evade capture indefinitely.
\end{proof}

The following theorem determines the localization number of the remaining cases.

\begin{theorem}\label{one}
If $G$ is a Moore graph of diameter $2$ that is $k$-regular with $k \geq 5$, then $\zeta(G)$ is one of $k-1$ or $k$.
\end{theorem}

\begin{proof}%[Proof of Theorem~\ref{one}]
We show that $\zeta(G) \leq  k$ by giving a winning strategy with $k$ cops.

To initialize, we show that in a small number of moves, the cops are able to identify that the robber is either on a set of $k-1$ vertices that are in $N(v)$, for some vertex $v$.
Let $x \in V(G)$ and $y,z \in N(x)$.
We place $k-1$ cops on the vertices of $N(x)\setminus \{y\}$ and one cop on a vertex $w \in N(z)\setminus \{x\}$.
If all cops on $N(x)$ probe $1$, then the robber is identified to be on $x$.
If only one cop on $v\in N(x)$ probes $1$ and the cop on $w$ probes $2$, then the robber is identified to be on a set of $k-1$ vertices of $N(v)$.
If no cop on $N(x)$ probes one and the cop on $w$ probes $2$, then the robber is identified to be on a set of $k-1$ vertices of $N(y)$ or on $y$.
In all other cases, the robbers location is found exactly.

Assuming that the robber is identified to be on either vertex $y$ or in a set of $k-1$ vertices that are in $N(y)$, the cops play on the $k-1$ vertices of $N(y) \setminus \{u\}$, where $u$ is the unique vertex of $N(y)$ that is known to not contain the robber. We also place a cop on a vertex in $N(z) \setminus \{y\}$, where $z \in N(y) \setminus \{u\}$. By an identical analysis to what was just performed above, we can find that the cops are able to identify that the robber is on a set of $k-1$ vertices that are in a $N(v)$, for some $v \in N(y)$, or is caught.

%If a cop probes a distance of $0$, then the robber is captured.
%If all cops probe a distance of $1$, then the robber is captured on $x$.
%If exactly one cop $C$ probes a distance of $1$, then the robber is found to reside on the $k-1$ vertices of $N(C) \setminus \{x\}$.
%If all cops probe a distance of $2$, then the robber is found to reside on one of the $k$ vertices of $N[y]\setminus \{x\}$.
%
%In this latter case where the robber is found to reside on the $k$ vertices of $N[y]\setminus \{x\}$, the robber moves, and then the cops play on the $k-1$ vertices of $N(y)\setminus \{x\}$.
%If all the cops probe a distance of $2$, then the robber is on $x$.
%If all the cops probe a distance of $1$, then the robber is on $y$.
%If a cop probes a distance of $0$, then the robber is found.
%Otherwise, a cop $C$ probes a distance of $1$, and we can identify that the robber resides on the $k-1$ vertices of $N(C) \setminus \{y\}$.
%%In the former case, if only one cop $C$ probes a distance of $1$, then the robber is on the $k-1$ vertices in $N(C)\setminus \{x\}$.
%Hence, in all cases, we can identify that the robber resides on $k-1$ vertices in a common neighborhood.

\begin{figure}[h!]
\centering
\includegraphics{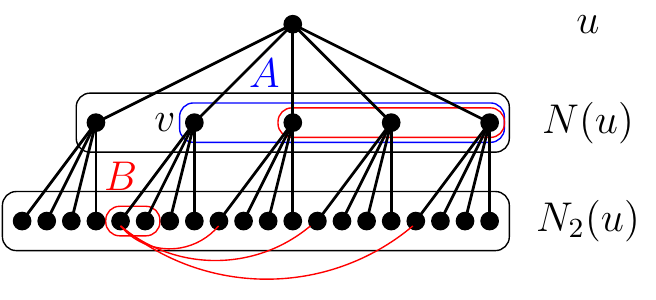}
\caption{The depiction of a (hypothetical) $5$-regular graph with diameter $2$ and girth $5$, with most of $u$'s second neighbors' edges omitted.}\label{fig:ecc2girth5Upper_4}
\end{figure}

We work inductively on $\alpha$, starting from $\alpha=1$ and increasing until $\alpha = k-3$.
See Figure~\ref{fig:ecc2girth5Upper_4} for a figure depicting this situation, with most of the edges between the second neighbors of $u$ omitted.
Suppose that the cops discover that for some $u \in V(G)$, the robber is on a subset $A \subseteq N(u)$ of cardinality $k- \alpha$.
The robber moves.
Let $v \in A$.
Select $\alpha+1$ vertices in $N(v) \setminus \{u\}$, which we label as $B$.
The cops place $k-1-\alpha \geq 2$ cops on the $k-1-\alpha$ vertices of $A\setminus\{v\}$, and $\alpha+1\geq 2$ cops on $B$.
The cops can identify the exact robber location if the robber moved to $u$ (all cops on $A$ probe a distance of $1$, which is at least two distinct cops).
The cops can identify the exact robber location if the robber stayed on $v$ (all cops in $B$ probe a distance of $1$, which is at least two distinct cops).
The cops can identify the exact robber location if the robber stayed on $A\setminus \{v\}$ (a cop in $A$ probes a distance of $0$).
The cops can identify if the robber moved to the set of $k-1-\alpha$ vertices of $B$ that does not contain a cop (all cops probe a distance of $2$).
If a cop $C_1$ on $A$ probes a distance of $1$ and another cop on $A$ probes a distance of $2$, then the robber is on $N(C_1) \setminus \{u\}$.
Further, if another cop $C_2$ on $B$ probes a distance of $1$, then the exact location of the robber is the unique vertex in the intersection $N(C_1) \cap N(C_2)$.
If instead all $\alpha+1$ cops on $B$ probe a distance of $2$, then the robber is on the set of $k-2-\alpha$ vertices of $N(C_1) \setminus \{u\}$ not adjacent to the $\alpha$ cops on $B$.
Hence in all cases, if the robber is not found, then the cops discover that for some $u' \in V(G)$, the robber is on a subset $A' \subseteq N(u')$ of cardinality at most $k-1- \alpha$.

Applying the induction yields a situation where, if the robber is not yet captured, the cops discover that for some $u \in V(G)$, the robber is on a subset $\{a_1,a_2\} \subseteq N(u)$.
The cop player places two cops on $N(a_1)\setminus\{u\}$ and $k-3$ cops on $N(a_2)\setminus\{u\}$, such that none of the cops are adjacent.
We may place the cops such that they are not adjacent, as the two cops on $N(a_1)\setminus\{u\}$ are adjacent to exactly two vertices of $N(a_2)\setminus\{u\}$, leaving $k-3$ vertices in $N(a_2)\setminus\{u\}$ not adjacent to the two cops.
There is one as-yet unplaced cop and we place the final cop on $u$.
If the cop on $u$ probes a distance of $0$, then the robber is on $u$.
If both cops on $N(a_1)\setminus\{u\}$ probe a distance of $1$, then the robber is on $a_1$.
If all cops on $N(a_2)\setminus\{u\}$ probe a distance of $1$, then the robber is on $a_2$.
If only one cop $C$ in $N(a_2)\setminus\{u\}$ probes a distance of $1$, then the robber is on the unique vertex of $N(a_1)\setminus\{u\}$ that is adjacent to $C$.
If only one cop $C$ in $N(a_1)\setminus\{u\}$ probes a distance of $1$, then the robber is on the unique vertex of $N(a_2)\setminus\{u\}$ that is adjacent to $C$.
All situations are covered; otherwise, it would imply the existence of a $4$-cycle, so the robber's exact location is found. The upper bound follows.

To complete the proof, we derive the lower bound.
We play with $k-2$ cops and show that the robber can evade capture.
Suppose that the robber was on some vertex $u$, and then the robber took his turn by moving or remaining on $u$.
We show that $N(u)$ contains two vertices that are indistinguishable by the cops, and hence, the robber has a starting move.
Every cop that will be played is either on $u,$ on $N(u)$, or on a neighbor of $N(u).$ If a cop is on $u,$ then they cannot distinguish between the points on $N(u).$  If a cop is on $N(u)$, then the cop has distance $0$ from one vertex in $N(u)$ and distance $2$ from all other vertices in $N(u)$. Finally, if  a cop is on a neighbor of some $v \in N(u)$, then that cop has distance $1$ from one vertex in $N(u)$ and distance $2$ from all other vertices in $N(u)$. Hence, no matter how we place the $k-2$ cops, only $k-2$ vertices of $N(u)$ can be uniquely distinguished from each other, leaving two vertices that cannot be distinguished by the cops, and the robber chooses one of these two to move to.
As long as the robber can play an initial move, then this implies that the robber may avoid capture indefinitely.
During the initial move, the robber chooses a vertex $u$ that cannot be distinguished by the cops. \end{proof}

By Theorem~\ref{one}, the localization number of the Hoffman-Singleton graph is either $6$ or $7$.

\subsection{Polarity Graphs} \label{sec:polarity}

Fix $q$ a prime power.
For a given projective plane $\mathrm{PG}(2,q)$ with points $P$ and lines $L$, a \emph{polarity} $\pi :P \rightarrow L$ is a bijection mapping points to lines such that $v\in \pi(u)$ whenever $u\in\pi(v)$.
The \emph{polarity graphs} are formed on vertex set $P$ by joining distinct vertices $u$ and $v$ if $u\in \pi (v)$ and $u \neq v$.
Polarity graphs have $q^2+q+1$ vertices.
The vertices $u$ with $u\in \pi(u)$ are called \emph{absolute} vertices, and have degree $q$, while all other points have degree $q+1$.
Baer~\cite{b} showed that there must be at least $q+1$ absolute vertices.
Polarity graphs are without 4-cycles~\cite{bef}, have $q(q+1)^2/2$ edges, have diameter $2$, and possess unbounded chromatic number
as $q\rightarrow \infty$~\cite{godsil}.

Those polarity graphs with exactly $q+1$ absolute vertices arise from an \emph{orthogonal} polarity (which exists for all $\mathrm{PG}(2,q)$), and such graphs are known as  the Erd\H{o}s-R\'{e}nyi graphs, written $\mathrm{ER}(q)$.
The \emph{Erd\H{o}s-R\'{e}nyi graphs}, have vertices as the points of $\mathrm{PG}(2,q),$ and $u$ is
adjacent to $v$ if $u^{T}v=0,$ where we identify vertices with 1-dimensional
subspaces of $\mathrm{GF}(q)^{3}.$
These are well-known examples of graphs
which are $C_{4}$-\emph{free} \emph{extremal}, in the sense that they
have the largest possible number of edges in a $C_{4}$-free graph on $q^2+q+1$ vertices; see \cite{brown,erdos}.
For more on polarity graphs, see \cite{mub}.

Polarity graphs were studied for the game of Cops and Robbers in~\cite{BB}, where bounds were given on the cop number. We provide lower and upper bounds on the metric dimension and localization number of polarity graphs.

\begin{theorem}\label{thm:dim_polarity}
If $G$ is a polarity graph with order $q^2+q+1$, then \[2q-5 \leq \beta(G) \leq 2q-1.\]
\end{theorem}
\begin{proof}
We first show the lower bound. Let $\alpha = \beta(G)$.
We aim to show that $2q-6$ cops are insufficient to capture the robber on the first round of play in the localization game.
We will assume that the cops pick a minimum sized resolving set of $G$, and analyze the distance vector for each $v \in V(G)$.
To each vertex in the graph, we assign the distance vector that would result if the robber chose this vertex.
Note that these distance vectors will contain only the symbols $0$, $1$, and $2$.
Exactly $\alpha$ vertices will have a distance vector that contains the symbol $0$, which are those vertices in the resolving set.
There are $q^2+q+1-\alpha$ remaining vertices that must be resolved by the resolving set.
At most one vertex will have distance $2$ to all cops in the resolving set, as if there were two, they would be indistinguishable.
Each vertex of the resolving set will have distance $1$ to at most $q+1$ vertices in $V(G)$.
The total number of occurrences of $1$ within all of the distance vectors must be exactly $(q+1) \alpha$.

At most $\alpha$ vertices will have distance $1$ to one cop and distance $2$ to all other cops.
We may assume that exactly  $\alpha$ vertices will have distance $1$ to one cop and distance $2$ to all other cops, as this is strictly beneficial to the cops under this argument.
There are $q \alpha$  occurrences of $1$ within the remaining distance vectors; that is, the distance vectors for vertices that have distance $1$ to at least two vertices in the resolving set.
As a result, there can be at most $q \alpha /2$ vertices with distance $1$ to at least two vertices.
This implies that there are at most $1+\alpha + q \alpha /2$ vertices distinguished by the $\alpha$ cops.
As we found before, there are $q^2+q+1-\alpha$ vertices still to be distinguished, and so we must have $q^2+q+1-\alpha \leq 1+\alpha + q \alpha /2$.
We find after rearranging that $2q - 6 +24/(q+4) \leq \alpha$, from which the lower bound of the result follows.

To show the upper bound, let $u$ be a vertex of degree $q$.
There must be $q^2$ vertices in the second neighborhood of $u$.
For this reason, each vertex in $N(u)$ must have degree $q+1$ and no $C_3$ can contain $u$, as each vertex in $N(u)$ must correspond to a set of $q$ vertices in the second neighborhood.
An equivalent way to say that $u$ is not contained in a $C_3$ is that $N[v_1] \cap N[v_2] = \{u\}$ for any $v_1,v_2 \in N(u)$.
Let $v \in N(u)$.
We place a cop on each of the vertices $(N(u)\cup N(v)) \setminus \{u,v\}$.
We claim that the cops can identify the robber's location during the first probe. If the robber is on $u$, then all cops on $N(u) \setminus \{v\}$ probe a distance of $1$ to the robber.
If the robber is on $v$, then all cops on $N(v) \setminus \{u\}$ probe a distance of $1$ to the robber.
If the robber is located on $(N(u)\cup N(v)) \setminus \{u,v\}$, then some cop probes a distance of $0$ to the robber.
Otherwise, the robber has distance $2$ to $u$ and distance $2$ to $v$.
Note that there are $q^2-q$ vertices in $A = V \setminus (N(u) \cup N(v))$.

There are $q^2-q = q(q-1)$ paths of length two from vertices in $N(u)\setminus\{v\}$ to vertices in $N(v)\setminus\{u\}$, and these cannot intersect $u$ or $v$. As such, all such paths will intersect with a vertex in $A$.
Further, if some vertex in $A$ is contained in two such paths, then a 4-cycle exists in the graph.
So by a pigeon-hole argument, each vertex in $A$ is contained in exactly one path of length two from $ N(u)\setminus\{v\}$ to $N(v)\setminus\{u\}$, say from vertex $b_1$ to vertex $b_2$.
The cops on $b_1$ and $b_2$ are the only ones that probe a distance of $1$ from the robber if the robber is on the vertex in $N(b_1) \cap N(b_2)$. Hence, the cops can uniquely identify the robber's position if the robber is on $A$.
The upper bound now follows.
\end{proof}

Theorem~\ref{thm:dim_polarity} provides us an upper bound for the localization number of a polarity graph. We also derive a lower bound.

\begin{theorem}\label{finalt}
If $G$ is a polarity graph with order $q^2+q+1,$ then \[\frac{2q-5}{3} \leq \zeta(G) \leq 2q-1.\]
\end{theorem}
\begin{proof}
For the lower bound, suppose for a contradiction that $\zeta(G) < (2q-5)/3$ and that the cops can capture the robber. We first show the inductive step.
To show that the robber always has a move that he can make during each turn, suppose that there are two vertices that the cops cannot distinguish while probing on their previous turn, $u_1$ and $u_2$.
Let $A_1 = N(u_1) \setminus N[u_2]$ and $A_2 = N(u_2) \setminus N[u_1]$.
We will assume that the robber only moves to a vertex in $A_1 \cup A_2$, which only serves to weaken the robber's strategy.
As the graph contains no 4-cycle, $|N(u_1) \cap N(u_2)|\leq 1$.
As a result,  $|A_1| \geq q-2$ and $|A_2| \geq q-2$ and so $|A_1 \cup A_2| \geq 2q-4$.

If a cop is placed on $u_1$ or $u_2$, then the cops will have distance $1$ to all vertices in $A_1$ and distance $2$ to all vertices in $A_2$, or vice versa.  Any cop that is not on $u_1$ or $u_2$ will have distance $1$ to at most one vertex of $A_1$ and at most one vertex of $A_2$, and distance $2$ to all other vertices of $A_1 \cup A_2$.
In addition to this, any cop on $A_1\cup A_2$ will have distance $0$ to exactly one vertex of $A_1 \cup A_2$.

For now, we suppose that there is a cop on $u_1$.
In this case, placing a cop on $u_2$ provides no additional information.
If a cop that is placed on a vertex $u_3 \in A_1\cup A_2$ directly helps to capture the robber, then either the cop on $u_3$ probed $0$ and the robber is identified to be on $u_3$; or the cops on $u_1$ and $u_3$ probed $1$ and the robber is identified to be on the unique vertex of $A_1$ of distance $1$ from $u_3$; or the cop on $u_3$ probed $1$ and the cop on $u_1$ probed $2$, and so the robber is identified to be on the unique vertex of $A_2$ of distance $1$ from $u_3$.

As a result, if we have $\alpha$ cops not in $\{u_1,u_2\}$, then there are $\alpha$ vertices of $A_1 \cup A_2$ that can be immediately resolved as they each have a cop of distance $0$ from them, and up to a further $2\alpha$ vertices that can be resolved as they have distance $1$ to a cop, and at most one vertex that can be resolved as it has distance $2$ to all cops not on $u_1$ or $u_2$. This implies that at most $3\alpha+1$ vertices of $A_1 \cup A_2$ can be resolved.
We must then have that $3\alpha+1 \geq 2q-4$, which is a contradiction as we have assumed there are less than $(2q-5)/3$ cops, and we have just shown that we require at least $\alpha+1$ cops.
Now note that if we do not have a cop on $u_1$ (or likewise $u_2$), then we likewise require more than $\alpha$ cops, and obtain the same contradiction. For the initial case, the robber considers two arbitrary vertices $u_1$ and $u_2.$ By the above analysis, there is a vertex in $N[u_1] \cup N[u_2]$ that cannot be distinguished by the cops. \end{proof}

\end{document}